\documentclass[10pt,leqno]{amsart}
\usepackage{indentfirst, amssymb,amsmath,amsthm,fix-cm,graphicx,latexsym}
\newtheoremstyle{case}{}{}{}{}{}{:}{ }{}
\newtheoremstyle{subcase}{}{}{}{}{}{:}{ }{}

\newtheorem*{theoA}{Theorem A}
\newtheorem*{theoB}{Theorem B}
\newtheorem*{theoC}{Theorem C}
\newtheorem*{theoD}{Theorem D}
\newtheorem*{theoE}{Theorem E}
\newtheorem*{theoF}{Theorem F}

\newtheorem{theo}{Theorem}[section]
\newtheorem{lem}{Lemma}[section]
\newtheorem{cor}{Corollary}[section]

\newtheorem{exm}{Example}[section]
\newtheorem{defi}{Definition}[section]

\newtheorem{ques}{Question}[section]
\newcommand{\ol}{\overline}
\newcommand{\ul}{\underline}

\numberwithin{subcase}{case}
\newcommand{\be}{\begin{equation}}
\newcommand{\ee}{\end{equation}}
\newcommand{\beas}{\begin{eqnarray*}}
\newcommand{\bea}{\begin{eqnarray}}
\newcommand{\eea}{\end{eqnarray}}
\newcommand{\eeas}{\end{eqnarray*}}
\newcommand{\lra}{\longrightarrow}
\newcommand{\bd}{\begin{doublespacing}}
\newcommand{\ed}{\end{doublespacing}}
\usepackage{amssymb,setspace,amsbsy,indentfirst}
\numberwithin{equation}{section}
\renewcommand{\vline}{\mid}
\begin{document}
\title[Two and three shared sets uniqueness problem...]{Two and three shared sets uniqueness problem on restricted and unrestricted ordered meromorphic functions}
\numberwithin {equation}{section}
	\date{}
	\author{Abhijit Banerjee\;\;\; And \;\;\;Saikat Bhattacharyya}
	\date{}
	\address{Abhijit Banerjee, Department of Mathematics, University of Kalyani, West Bengal 741235, India.}
	\email{abanerjee\_kal@yahoo.co.in, abanerjeekal@gmail.com}
	\address {*Saikat Bhattacharyya, Department of Science and Humanities, Jangipur Government Polytechnic, West Bengal 742225, India.}
	\email{*saikat352@gmail.com, saikatbh89@yahoo.com}
	
	\maketitle
	\let\thefootnote\relax
	\footnotetext{2000 Mathematics Subject Classification: 30D35.}
	\footnotetext{Key words and phrases: Meromorphic functions, uniqueness, order, shared sets.}
	\footnotetext{Type set by \AmS -\LaTeX}
	\begin{abstract}
In this paper we have mainly focused on the uniqueness of two meromorphic functions under the realm of two shared sets problem with certain restrictions which in turn improve the results of \cite{AUSM_SahKar_18} and \cite{BSMM_SahSar_19}. In particular, pointing out some gaps in the statements of the results in \cite{AUSM_SahKar_18} and \cite{BSMM_SahSar_19} we have removed the order restrictions on the main results in \cite{AUSM_SahKar_18} and \cite{BSMM_SahSar_19} at the cost of sharing of an additional singleton. A number of examples have been provided by us to show some conditions in our results are essential. 
	   \end{abstract}
	\section{Introduction Definitions and Results}
	Throughout the paper, the term ``meromorphic" will be used to mean meromorphic in the whole complex plane $\mathbb{C}$. 
	Let $f$ and $g$ be two non-constant meromorphic functions defined in the open complex plane $\mathbb{C}$. 
	We denote by $T(r)=\max\{T(r,f), T(r,g)\}$. Usually, $S(r,f)$ denotes any quantity satisfying $S(r,f)= o(T(r,f))$ and so we use the notation $S(r)$ to denotes any quantity satisfying $S(r)= o(T(r))$ as $r \rightarrow \infty$, outside a possible exceptional set of finite linear measure. \par
	For some $a\in\mathbb{C}$, we denote by $E(a;f)$, the collection of the zeros of $f-a$, where a zero is counted according to its multiplicity. In addition to this, when $a=\infty$, the above definition implies that we are considering the poles. In the same manner, by $\ol E(a;f)$, we denote the collection of the distinct zeros or poles of $f-a$ according as $a\in\mathbb{C}$ or $a=\infty$ respectively. If $E(a;f)=E(a;g)$ we say that $f$ and $g$ share the value $a$ CM (counting multiplicities) and if  $\ol E(a;f)=\ol E(a;g)$, then we say that $f$ and $g$ share the value $a$ IM (ignoring multiplicities). 
	
	Let $S$ be a set of distinct elements of $\mathbb{C}\cup\{\infty\}$ and $E_{f}(S)=\cup_{a\in S}\{z: f(z)-a=0\}$, where each zero is counted according to its multiplicity. If we do not count the multiplicity then the set $\cup_{a\in S}\{z: f(z)-a=0\}$ is denoted by $\ol E_{f}(S)$.\par
	If $E_{f}(S)=E_{g}(S)$ we say that $f$ and $g$ share the set $S$ CM. On the other hand if $\ol E_{f}(S)=\ol E_{g}(S)$, we say that $f$ and $g$ share the set $S$ IM.\par
	The gradation of sharing known as weighted sharing is defined in the following way: 
		\begin{defi} \cite{NMJ_Lah_01, CVTA_Lah_01} Let $k$ be a non-negative integer or infinity. For $a\in\mathbb{C}\cup\{\infty\}$ we denote by $E_{k}(a;f)$ the set of all $a$-points of $f$, where an $a$-point of multiplicity $m$ is counted $m$ times if $m\leq k$ and $k+1$ times if $m>k$. If $E_{k}(a;f)=E_{k}(a;g)$, we say that $f, \; g$ share the value $a$ with weight $k$ and denote it by $(a,k)$. The IM and CM sharing corresponds to $(a,0)$ and $(a,\infty)$ respectively.\end{defi}
		
		\begin{defi} \cite{NMJ_Lah_01} Let $S$ be a set of distinct elements of $\mathbb{C}\cup\{\infty\}$ and $k$ be a non-negative integer or $\infty$. We denote by $E_{f}(S,k)$ the set $\cup_{a\in{S}}E_{k}(a;f)$.
			Clearly $E_{f}(S)=E_{f}(S,\infty)$ and $\ol E_{f}(S)=E_{f}(S,0)$. \end{defi}
	\begin{defi} \cite{CVTA_Lah_01}
		We denote by $N_2(r,a;f)$ the sum $\ol N(r,a;f) + \ol N(r,a;f| \geq 2)$.\end{defi}
	We denote by $\mathcal{M}(\mathbb{C})$ the class of all meromorphic functions defined in $\mathbb{C}$ and by $\mathcal{M}_1(\mathbb{C})$ the class of meromorphic functions which have finitely many poles in $\mathbb{C}$.
	The order $\lambda(f)$ of $f \in \mathcal{M}(\mathbb{C})$ is defined as \beas \lambda(f) = \limsup_{r \to \infty} \frac{\log T(r,f)}{\log r}. \eeas 
	\par In 1982, employing the order notion on entire functions  sharing two sets, Gross-Osgood \cite{LNPAM_GroOsd_82} found the relation between them. 
	\begin{theoA} \cite{LNPAM_GroOsd_82}
		Let $S_1=\{1,-1\}$, $S_2=\{0\}$. If $f$ and $g$ are non-constant entire functions of finite order such that $E_f(S_j,\infty)= E_g(S_j,\infty)$ for $j=1,2$, then $f = \pm g$ or $fg =1.$
	\end{theoA}

	Next we give the following definitions.
	\begin{defi}
		Let $r$ be a positive integer and $S_1 = \{\alpha_1,\alpha_2,\ldots,\alpha_r\} \subset \mathbb{C}\setminus \{0
		\}$.\par
Let $ a_{r-1}, \ldots, a_1$ and $C\not=0$ be complex numbers.
We define \bea \label{e1.1} U(z)= CzQ(z),\eea  where
\beas  zQ(z)&=&(z-\alpha_1)(z-\alpha_2)\ldots(z-\alpha_r)+(-1)^{r+1}\alpha_1\alpha_2\ldots\alpha_r\\&=&z(z^{r-1}+ a_{r-1}z^{r-2}+ \ldots + a_2 z+ a_1)\nonumber\\&=&z(z- \gamma_1)\ldots(z-\gamma_{m_1})(z-\gamma_{m_1+1})^{n_{m_1+1}}\ldots (z-\gamma_{m_1+m_2})^{n_{m_1+m_2}},\eeas $C= \frac{1}{(-1)^{r+1}\alpha_1 \alpha_2 \ldots \alpha_r}$; $\alpha_1, \alpha_2, \ldots , \alpha_r$ be distinct roots of the equation $U(z)-1=0$, $\gamma_i's\; (i=1,2,\ldots,m_1+m_2)$ be distinct complex constants and $n_i$'s $(\geq 2)$ associated with each factor $(z-\gamma_i)$, represents its multiplicity for $i=m_1+1,m_1+2,\ldots, m_1+m_2$. \end{defi}\par 
We denote $m_1(m_2)$ and $m_1^*(m_2^*)$ as the number of simple(multiple) zeros of $Q(z)$ and $U(z)$ respectively. Then we define $\Gamma_1 := m_1+ m_2$, $\Gamma_1^* := m_1^*+ m_2^*$, $\Gamma_2:=m_1+2m_2$ and $\Gamma_2^*:=m_1^*+2m_2^*$.\par
For the polynomial $U(z)$ let us define the following two functions:
 \[
 \chi_0^{r-1} = 
 \begin{cases}
 0, & if\; a_1 \not=0, \\
 1, & if \;a_1 =0
 \end{cases}
 \]
 and 
 \[
 \mu_0^{r-1} = 
 \begin{cases}
 0, & if\; a_1 \not=0\; or \;a_2 \not =0, \\
 1, &otherwise.
 \end{cases}
 \]
 \par In view of the above definition, we see that $\Gamma_1^*=\Gamma_1-\chi_0^{r-1}+1$ and $\Gamma_2^*= \Gamma_2-\mu_0^{r-1}+1$.

In 2017, Chen \cite{OpenM_Chen_17} investigated the uniqueness of a special class of meromorphic functions with restricted order as follows: 
	\begin{theoB} \cite{OpenM_Chen_17}
Let $r$ be a positive integer and let $S_1 = \{\alpha_1,\alpha_2,\ldots,\alpha_r\}$, $S_2= \{\beta_1, \beta_2\}$, where $\alpha_1, \alpha_2, \ldots, \alpha_r, \beta_1, \beta_2$ are $r+2$ distinct finite complex numbers satisfying \bea \label{e1.2} (\beta_1-\alpha_1)^2(\beta_1-\alpha_2)^2\ldots (\beta_1-\alpha_r)^2\not= (\beta_2-\alpha_1)^2(\beta_2-\alpha_2)^2\ldots (\beta_2-\alpha_r)^2.\eea
If two non-constant meromorphic functions $f$ and $g$ in $\mathcal{M}_1(\mathbb{C})$ share $(S_1, \infty)$, $(S_2,0)$ and if the order of $f$ is neither an integer nor infinite, then $f\equiv g$.
\end{theoB}
The following examples show that when the order of the function $f$ is an integer or infinite, {\it Theorem B} cease to be hold.
\begin{exm} \label{ex1}
Let $f(z)= e^z$ and $g(z)= \alpha e^{-z}$, $\alpha \in \mathbb{C} \setminus \{0, \frac{1}{n}, 1\}$ and $n(\not = 0,1)$ be an integer. Clearly $f$ and $g$ share the sets $S_1= \{z: nz^2-(n^2\alpha+1)z+n\alpha=0\}$ and $S_2=\{1, \alpha\}$. Also we see that $S_1=\{n\alpha, \frac{1}{n}\}$ and so for $\alpha \not = \frac{1}{n}, 1$ 
\beas (1- n \alpha)^2\bigg(1-\frac{1}{n}\bigg)^2\not= (\alpha-n\alpha)^2\bigg(\alpha-\frac{1}{n}\bigg)^2.\eeas  
Here we note that $\lambda(f)=1$.  
\end{exm} 
\begin{exm} \label{ex2}
Let $f(z)= e^{sin z}$ and $g(z)=\alpha e^{-sin z}$, $\alpha \in \mathbb{C} \setminus \{0, \frac{1}{n}, 1\}$ and $n(\not = 0,1)$ be an integer. Proceeding in the similar way as done in {\it Example \ref{ex1}} we can show that $f$ and $g$ share the sets $S_1$ and $S_2$. Here we note that $\lambda(f)$ is infinite.
\end{exm}
In 2018, with some bound on the cardinality of the set $S_1$, Sahoo and Karmakar \cite{AUSM_SahKar_18} were able to relax the nature of sharing of the set $S_1$ in {\it Theorem B}. Sahoo and Karmakar's \cite{AUSM_SahKar_18} result was as follows:
\begin{theoC} \cite{AUSM_SahKar_18}
Let $S_1 \subset \mathbb{C}\setminus \{0\}$ and $S_2$ be defined as in {\it Theorem B} satisfying the condition (\ref{e1.2}) with $r>2\Gamma_2-2\mu_0^{r-1}+2$. Let $\{f,\; g\} \in \mathcal{M}_1(\mathbb{C})$ be two meromorphic functions, where  $f$ is of non-integer finite order, share $(S_1,2)$ and $(S_2, 0)$; then $f\equiv g$.
\end{theoC}
\begin{theoD} \cite{AUSM_SahKar_18}
Let $S_1 \subset \mathbb{C}\setminus \{0\}$ and $S_2$ be defined as in {\it Theorem B} satisfying the condition (\ref{e1.2}) with $r>2\Gamma_2-2\mu_0^{r-1}+2$. If $\mathcal{M}_2(\mathbb{C})$ denotes the subclass of meromorphic functions of non-integer finite order in $\mathcal{M}_1(\mathbb{C})$, then the sets $S_1$ and $S_2$ are the URS of meromorphic functions in  $\mathcal{M}_2(\mathbb{C})$.
\end{theoD}
In 2019, Sahoo and Sarkar \cite{BSMM_SahSar_19} again found almost the same result under IM sharing of the set $S_1$ as follows:
\begin{theoE} \cite{BSMM_SahSar_19}
Let $S_1 \subset \mathbb{C}\setminus \{0\}$ and $S_2$ be defined as in {\it Theorem B} satisfying the condition (\ref{e1.2}) with $r>2\Gamma_2+ 3\Gamma_1-3\chi_0^{r-1}-2\mu_0^{r-1}+5$. Let $\{f,\; g\} \in \mathcal{M}_1(\mathbb{C})$ be two meromorphic functions, where $f$ is of non-integer finite order, share $(S_1,0)$ and $(S_2, 0)$; then $f\equiv g$.
\end{theoE}
\begin{theoF} \cite{BSMM_SahSar_19}
Let $S_1 \subset \mathbb{C}\setminus \{0\}$ and $S_2$ be defined as in {\it Theorem B} satisfying the condition (\ref{e1.2}) with $r >2\Gamma_2+ 3\Gamma_1-3\chi_0^{r-1}-2\mu_0^{r-1}+5$. If $\mathcal{M}_2(\mathbb{C})$ denotes the subclass of meromorphic functions of non-integer finite order in $\mathcal{M}_1(\mathbb{C})$, then the sets $S_1$ and $S_2$ are the URS of meromorphic functions in  $\mathcal{M}_2(\mathbb{C})$.
\end{theoF}
We can see that both in {\it Theorem D} and {\it Theorem F}, a condition is needed for the conclusion to hold and so the theorems are redundant as the sets $S_1$ and $S_2$ are not URS according to the definition as stated in \cite{KMJ_LiYang_95}. The sets are not even BURSM according to the definition in \cite{NMJ_Ban_13}. Although the sets can be defined as restricted BURSM. \par  
 Considering the above theorems the following questions are  inevitable:\par 
 \begin{ques} \label{q1}
 What happens if we consider the functions $f$ and $g$ to be non-constant meromorphic functions, i.e., $f, \; g \in \mathcal{M}(\mathbb{C})$?
\end{ques}
\begin{ques} \label{q2}
Is it possible to remove the order restrictions in {\it Theorem C} ? 
\end{ques}
The motivation of writing the paper is to investigate the possible answer of the above questions. We have presented the
accurate form of {\it Theorems C and D} and extend the results.\par 
 The following theorem is one of the main results of the paper corresponding to  {\it Question \ref{q1}}.
\begin{theo} \label{t1}
Let $r$ be a positive integer, $S_1 = \{\alpha_1,\alpha_2,\ldots,\alpha_r\} \subset \mathbb{C}\setminus \{0\}$ and $S_2= \{\beta_1, \beta_2\}$ satisfy the condition (\ref{e1.2}); $\alpha_1, \alpha_2, \ldots, \alpha_r, \beta_1, \beta_2$ be finite complex numbers such that at most one of $\beta_i$ $(i=1,2)$ $\in S_1$. Let $f$, $g$ be two meromorphic functions, where $f$ is of non-integer finite order, share $(S_1,l)$ and $(S_2,0)$. \par 
If $l =2$ and \bea \label{t1.1}r &>& 2\Gamma_2-2\mu_0^{r-1}+6-4\min\{\Theta(\infty,f),\Theta(\infty,g)\},\eea
or $l=1$ and \bea \label{t1.2} r &>&\max\bigg\{2\Gamma_1-2\chi_0^{r-1}+4-\Theta(\infty,f)- \Theta(\infty,g),\nonumber\\&& 2\Gamma_2+\frac{\Gamma_1}{2}-\frac{1}{2}\chi_0^{r-1}-2\mu_0^{r-1}+7-\frac{9}{2} \min\{\Theta(\infty,f),\Theta(\infty,g)\} \bigg\},\eea
or $l=0$ and \bea \label{t1.3} r &>& \max \bigg\{2\Gamma_1-2\chi_0^{r-1}+4-\Theta(\infty,f)- \Theta(\infty,g),\nonumber\\&& 2\Gamma_2+3\Gamma_1-3\chi_0^{r-1}-2\mu_0^{r-1}+12-7\min\{\Theta(\infty,f), \Theta(\infty,g)\}\bigg\},\eea
then $f \equiv g$, provided $m_1 > 1$.
\end{theo}

\begin{cor} \label{c1}
Let $S_1\subset \mathbb{C}\setminus \{0\}$ and $S_2$ be defined as in {\it Theorem \ref{t1}} satisfying the condition (\ref{e1.2}). Let two meromorphic functions $\{f,\; g\}  \in \mathcal{M}_1(\mathbb{C})$, where $f$ is of non-integer finite order, share $(S_1,l)$ and $ (S_2,0)$. \par 
If $l =2$ and \bea \label{c1.1} r>2\Gamma_2-2\mu_0^{r-1}+2,\eea
or $l=1$ and \bea \label{c1.2} r >2 \Gamma_2 +\frac{\Gamma_1}{2}-\frac{1}{2}\chi_0^{r-1}-2 \mu_0^{r-1}+\frac{5}{2},\eea
or $l=0$ and \bea \label{c1.3} r > 2\Gamma_2+ 3\Gamma_1-3\chi_0^{r-1}-2\mu_0^{r-1}+5,\eea
then $f\equiv g$.
\end{cor}
Our next result will give an affirmative answer to {\it Question \ref{q2}} for a suitable choice of the set $S_2$ together with sharing of the value $0$. 
\begin{theo} \label{t2}
Let $S_1\subset \mathbb{C}\setminus \{0\}$ be defined as in {\it Theorem \ref{t1}}, $S_2=\{1,c\}$, where $c \in \mathbb{C}\setminus\{-1,\frac{1}{2},2\}$ and $S_3= \{0\}$, such that at most one of $\{1,c\} \in S_1$. Let two non-constant meromorphic functions $f$, $g$ share the sets $(S_1,2)$, $(S_2,0)$ and $(S_3, 0)$. If
\bea \label{t2.1}r &>& 2\Gamma_1-2\chi_0^{r-1}+7-4\min\{\Theta(\infty,f),\Theta(\infty,g)\}\eea and
\bea \label{t2.2} m_1>1, \eea
then $f\equiv g$, provided the condition 
\bea \label{t2.3} (1-\alpha_1)^2(1-\alpha_2)^2\ldots (1-\alpha_r)^2\not= (c-\alpha_1)^2(c-\alpha_2)^2\ldots (c-\alpha_r)^2\eea
 holds.
\end{theo}
\begin{cor} \label{c2}
Let $S_1\subset \mathbb{C}\setminus \{0\}$, $S_2$ and $S_3$ be defined as in {\it Theorem \ref{t2}} satisfying the condition (\ref{t2.3}). Let two transcendental meromorphic functions $f,\; g$ in $\mathcal{M}_1(\mathbb{C})$ share the sets $(S_1,2)$, $(S_2,0)$ and $(S_3, 0)$. If 
\bea \label{c2.1}r &>& 2\Gamma_1-2 \chi_0^{r-1}+3\eea
and 
\bea \label{c2.2} m_1> 0, \eea
then $f\equiv g$.
\end{cor}
The following examples show that the condition (\ref{e1.2}) in {\it Theorem \ref{t1}} and {\it Corollary \ref{c1}} cannot be dropped. 
\begin{exm} \label{ex4}
Let $f(z)= \sum\limits_{n=1}^{\infty} \frac{z^n}{n^{\alpha n}}$ and $g(z)= -f(z)$, where $\alpha(\not=0,1)$ is an integer. For a fixed positive integer $k$, let $S_1= \{-1, 1, -2, 2, \ldots, -k, k\}$ and $S_2=\{-(k+1), (k+1)\}$. Clearly, $f$, $g \in \mathcal{M}_1(\mathbb{C})$ and they share $S_1$ and $S_2$. Also using the result {\em\cite[p. 288]{Springer_Con_73}}, we note that \beas \lambda(f)=\frac{1}{\liminf\limits_{n \to \infty} \frac{log n^{\alpha n}}{n \log n}}= \limsup\limits_{n \to \infty} \frac{n \log n}{\log n^{\alpha n}}= \frac{1}{\alpha}. \eeas
Here the condition (\ref{e1.2}) does not hold.
\end{exm}
\begin{exm} \label{ex5}
Let $f(z)= \sum\limits_{n=1}^{\infty} \frac{z^n}{(n!)^{\alpha}}$ and $g(z)= -f(z)$, where $\alpha(\not=0,1)$ is an integer. Proceeding in the same way as done in {\it Example \ref{ex4}} we can show that $f$, $g$ share the sets $S_1$ and $S_2$. Again by using the result {\em \cite[p. 288]{Springer_Con_73}}, we note that  \beas \lambda(f)=\frac{1}{\liminf\limits_{n \to \infty} \frac{log (n!)^{\alpha}}{n \log n}}= \limsup_{n \to \infty} \frac{n \log n}{\log (n!)^{\alpha }}= \frac{1}{\alpha}\limsup_{n \to \infty} \frac{n \log n}{\log (n!)}= \frac{1}{\alpha}. \eeas
Here again the condition (\ref{e1.2}) does not hold.
\end{exm}
The following examples show that the order of the function $f$ cannot be an integer or infinite in {\it Corollary \ref{c1}}.
\begin{exm} \label{ex6}
Let $f(z)= e^z$ and $g(z)= e^{-z}$. Clearly $f$ and $g$ share the sets $S_1= \{z: z^{2n+1}=1\}$ and $S_2=\{i, -i\}$, where the sets $S_1$ and $S_2$ satisfy the condition $(\ref{e1.2})$. Here $\lambda(f)=1$. 
\end{exm} 
\begin{exm} \label{ex7}
Let $f(z)= e^{sin z}$ and $g(z)= e^{-sin z}$. Clearly in the similar way as in {\it Example \ref{ex6}} we can show that $f$ and $g$ share the sets $S_1$ and $S_2$. Here we note that $\lambda(f)$ is infinite.
\end{exm}
The following example shows that the condition (\ref{c2.2}) is sharp in Corollary {\ref{c2}}.
\begin{exm}
Let $f(z)= e^z$ and $g(z)=ce^{-z}$, where $c \in \mathbb{C}\setminus\{-1, 0,\frac{1}{2},2\}$. Clearly for $\alpha^4=1$ $f$, $g$ share the sets $S_1=\{\sqrt{c}, \sqrt{c}\alpha, \sqrt{c}\alpha^2, \sqrt{c}\alpha^3\}$, $S_2= \{1,c\}$ and $S_3=\{0\}$. Here $m_1=0(\not > 0)$ and also we note that 
$ (1-\sqrt{c})^2(1-\sqrt{c}\alpha)^2(1-\sqrt{c}\alpha^2)^2(1-\sqrt{c}\alpha^3)^2\not= (c-\sqrt{c})^2(c-\sqrt{c}\alpha)^2(c-\sqrt{c}\alpha^2)^2(c-\sqrt{c}\alpha^3)^2.$

\end{exm}
The next examples show respectively that at most one of the elements in the set $S_2$ in {\it Theorems \ref{c1} and \ref{c2}} has to be distinct from the elements of the set $S_1$ is essential. In other words when $S_2\subset S_1$, then conclusion of {\it Theorem \ref{c1}} and {\it Corollary \ref{c2}} does not hold respectively.
\begin{exm}
Let $f$ be defined as in {\it Example \ref{ex5}} and $g(z)=\frac{1}{f(z)}$. Then clearly $f$, $g$ share the sets $S_1=\{\alpha_1, \frac{1}{\alpha_1}, \alpha_2, \frac{1}{\alpha_2}, \ldots, \alpha_r, \frac{1}{\alpha_r}\}$ and $S_2=\{\alpha_r, \frac{1}{\alpha_r}\}$, where $\alpha_i$ are constants for $i=1,2,\ldots, r$. Here $S_2 \subset S_1$.
\end{exm}
\begin{exm} Let $f(z)= e^z$ and $g(z)=e^{-z}$. Clearly for $6b=a(b^6-1)$ and $b^4-4b^2+5=0$ $f$, $g$ share the sets $S_1=\{z: z(z+a)(z+b)^6+1=0\}$, $S_2=\{1,-1\}$ and $S_3=\{0\}$. Here $m_1=1(>0)$. We can write the set $S_1=\{1, -1,\alpha_1, \frac{1}{\alpha_1}, \alpha_2, \frac{1}{\alpha_2}, \alpha_3, \frac{1}{\alpha_3}\}$, where $\alpha_i(\not=0)$ are constants for $i=1,2,3$ and so we note that $S_2\subset S_1$. \end{exm}  
For the standard definitions and notations of the value distribution theory we refer to \cite{Clarendon_Hay_64}. But in the paper we have used some more notations and definitions which are explained below.
\begin{defi} \cite{IJMS_Lah_01} For $a\in\mathbb{C}\cup\{\infty\}$ and for a positive integer $m$ we denote by $N(r,a;f\vline\leq m) (N(r,a;f\vline\geq m))$ the counting function of those $a$-points of $f$ whose multiplicities are not greater(less) than $m$ where each $a$-point is counted according to its multiplicity.
		
$\ol N(r,a;f\vline\leq m) (\ol N(r,a;f\vline\geq m))$ are defined similarly, where in counting the $a$-points of $f$ we ignore the multiplicities.
		
Also $N(r,a;f\vline <m)$, $N(r,a;f\vline >m)$, $\ol N(r,a;f\vline <m)$ and $\ol N(r,a;f\vline >m)$ are defined analogously.  \end{defi}
	\begin{defi} \cite{KMJ_Yi_99} Let $f$ and $g$ be two non-constant meromorphic functions such that $f$ and $g$ share $(a, 0)$. Let $z_0$ be an $a$-point of $f$ with multiplicity $p$, an $a$-point of $g$ with multiplicity $q$. We denote by $\ol N_L(r,a;f)$ the reduced counting of those $a$-points of $f$ and $g$ where $p>q$, by $N_E^{1)}(r,a;f)$ the counting function of those $a$-points of $f$ and $g$ where $p=q=1$, by $\ol N_E^{(2}(r,a;f)$ the reduced counting function of those $a$-points of $f$ and $g$ where $p=q \geq 2$. In the same way we can define $\ol N_L(r,a;g)$, $N_E^{1)}(r,a;g)$, $\ol N_E^{(2}(r,a;g)$. In a similar manner we can define  $\ol N_L(r,a;f)$ and $\ol N_L(r,a;g)$ for $a\in\mathbb{C}\cup\{\infty\}$.  \end{defi}
When $f$ and $g$ share $(a,m)$, $m \geq 1$, then $N_E^{1)}(r,a;f)= N(r,a;f \mid =1)$.
\begin{defi} \cite{NMJ_Lah_01, CVTA_Lah_01} 
	Let $f$, $g$ share a value $(a,0)$. We denote by $\ol N_*(r,a;f,g)$ the reduced counting function of those $a$-points of $f$ whose multiplicities differ from the multiplicities of the corresponding $a$-points of $g$. \par 
	Clearly, $\ol N_*(r,a;f,g)= \ol N_*(r,a;g,f)= \ol N_L(r,a;f)+ \ol N_L(r,a;g)$. 
\end{defi}
\section{Lemmas}
In this section we present some lemmas which will be needed in the sequel. Henceforth unless otherwise stated for the sake of simplicity for two non-constant meromorphic functions $f$ and $g$ we denote by \bea \label{e2.1} F=U(f), \; G=U(g),\eea where $U(z)$ is defined as in (\ref{e1.1}).\par We shall also denote by $H$ the following function: 
\bea \label{e2.2} H&=&\left(\frac{\;\;F^{''}}{F^{'}}-\frac{2F^{'}}{F-1}\right)-\left(\frac{\;\;G^{''}}{G^{'}}-\frac{2G^{'}}{G-1}\right).\eea

\begin{lem} \cite{MathZ_Yang_72} \label{l1}
Let $f$ be a non-constant meromorphic function and $P(f)= a_n f^n+ a_{n-1}f^{n-1}+ \ldots+ a_1 f+ a_0$, where $a_0, a_1, a_2, \ldots, a_n$ are constants and $a_n\not=0$. Then $T(r, P(f))= nT(r,f)+ O(1).$
\end{lem}

\begin{lem} \label{l2}
Let $S_1$, $S_3$ be defined as in {\it Theorem \ref{t2}} and $F$, $G$ be given by (\ref{e2.1}). If two non-constant meromorphic functions $f$, $g$ share $(S_1,0)$, $(S_3,0)$ and $H\not \equiv 0$, then
\beas N(r,H) &\leq& \ol N_*(r,0;f,g)+ \sum\limits_{i=1}^{\Gamma_1-\chi_0^{r-1}}\bigg\{\ol N(r, \gamma_i;f)+ \ol N(r, \gamma_i;g)\bigg\}+ \ol N_*(r,1; F,G)\\&& + \ol N(r,\infty;f)+ \ol N(r, \infty;g)+ \ol N_0(r,0;f')+ \ol N_0(r,0;g'),\eeas where $\ol N_0(r,0;f')$ is the reduced counting function of those zeros of $f'$ which are not the zeros of $f(f-\gamma_1)(f-\gamma_2)\ldots (f-\gamma_{m_1+m_2})(F-1)$ and $N_0(r,0;g')$ is similarly defined.    
\end{lem}
\begin{proof}
Since $f$, $g$ share the set $(S_1,0)$, it follows that $F$ and $G$ share $(1,0)$. We can easily verify that possible poles of $H$ occurs at (i) those zeros of $f$ and $g$ whose multiplicities are distinct from the multiplicities of the corresponding zeros of $g$ and $f$ respectively, (ii) zeros of  $(f-\gamma_i)$ and $(g-\gamma_i)$, for $i=1,2, \ldots, m_1+m_2$, (iii) poles of $f$ and $g$, (iv) those 1-points of $F$ and $G$ whose multiplicities are distinct from the multiplicities of the corresponding 1-points of $G$ and $F$ respectively, (v) zeros of $f'$ which are not the zeros of $f(f-\gamma_1)(f-\gamma_2)\ldots (f-\gamma_{m_1+m_2})(F-1)$, (vi) zeros of $g'$ which are not the zeros of $g(g-\gamma_1)(g-\gamma_2)\ldots (g-\gamma_{m_1+m_2})(G-1)$ . Since $H$ has only simple poles, clearly the lemma follows from the above explanations.
\end{proof}
\begin{lem} \cite{CVTA_Lah_01} \label{l3}
Let $F$, $G$ be two non-constant meromorphic function sharing $(1,2)$. Then one of the following cases holds:
\beas&(i)& T(r,F) \leq N_2(r,0; F)+ N_2(r,0; G) + N_2(r, \infty; F)+ N_2(r, \infty; G)\\&&+ S(r,F)+ S(r,G),\; the \; same \; inequality \; holds \; for \; T(r,G);\\
&(ii)& F\equiv G;\\ &(iii)& F \cdot G \equiv 1.\eeas 
\end{lem}
\begin{lem} \cite{IJMMS_Ban_05} \label{l4} 
Let $F$, $G$ be two non-constant meromorphic functions such that they share $(1,1)$ and $H \not \equiv 0$. Then \beas T(r, F) &\leq& N_2(r,0; F)+N_2(r,\infty; F )+N_2(r,0;G)+N_2(r,\infty;G) \\&& +\frac{1}{2}\ol N(r,0; F)+ \frac{1}{2}\ol N(r,\infty; F)+S(r,F)+S(r,G).\eeas	
\end{lem}
\begin{lem} \cite{IJMMS_Ban_05} \label{l5}
Let $F$, $G$ be two non-constant meromorphic functions such that they share $(1,0)$ and $H \not\equiv 0.$ Then 
\beas  T(r, F) &\leq& N_2(r,0; F)+N_2(r,\infty; F)+N_2(r,0;G)+N_2(r,\infty;G)+2 \ol N(r,0; F)
\\&&+2\ol N(r,\infty; F)+\ol N(r,0;G)+\ol N(r,\infty;G)+S(r,F)+S(r,G).\eeas	
\end{lem}
\begin{lem} \cite{CVTA_Yi_95} \label{l6}
If $H \equiv 0$, then $T(r,G)=T(r,F)+ O(1)$. If, in addition, \beas & &\limsup\limits_{r \lra \infty}\frac{\ol N(r,0;F)+ \ol N(r, 0; G)+ \ol N(r,\infty;F)+ \ol N(r, \infty;G)}{\max\{T(r,F), T(r,G)\}}< 1,\\& &r\not \in I \eeas where $I \subset (0,1)$ is a set of infinite linear measure, then either $F\equiv G$ or $F \cdot G \equiv 1$.
\end{lem}
\begin{lem} \label{l7}
Let $F$, $G$ be defined as in (\ref{e2.1}). Then 
\beas \ol N(r,0; F) &\leq& \ol N(r,0;f)+ (\Gamma_1-\chi_0^{r-1})T(r,f);\\
\ N_2(r,0;F) &\leq& (1+\chi_0^{r-1})\ol N(r,0;f)+ (\Gamma_2-\chi_0^{r-1}-\mu_0^{r-1})T(r,f).\eeas
Similar results hold for the function $G$.
\end{lem}
\begin{proof}
Here we have to consider two cases:\par
{\bf Case 1.} Suppose none of $\gamma_i's\; (i=1,2,\ldots,m_1+m_2)$ is zero. Then 
\beas \ol N(r,0; F) \leq \ol N(r,0;f)+ \sum\limits_{i=1}^{m_1+m_2}\ol N(r, \gamma_i;f) \leq \ol N(r,0;f)+ \Gamma_1 T(r,f);\eeas
\beas N_2(r,0; F) &\leq& \ol N(r,0;f) + \sum\limits_{i=1}^{m_1} \ol N(r, \gamma_i;f)+2\sum\limits_{i=m_1+1}^{m_1+m_2} \ol N(r, \gamma_i;f) \\&\leq& \ol N(r,0;f)+ \Gamma_2T(r,f). \eeas
{\bf Case 2.} Next let one of $\gamma_i's\; (i=1,2,\ldots,m_1+m_2)$ is zero.\par 
{\bf Subcase 2.1.} Suppose one among $\gamma_i's\; (i=1,2,\ldots,m_1)$ is zero. Without loss of generality let us assume that $\gamma_1=0$. Then 
\beas \ol N(r,0; F) \leq \ol N(r,0;f)+ \sum\limits_{i=2}^{m_1+m_2}\ol N(r, \gamma_i;f) \leq  \ol N(r,0;f)+(\Gamma_1-1) T(r,f);\eeas
\beas N_2(r,0; F) &\leq& 2 \ol N(r,0;f) + \sum\limits_{i=2}^{m_1} \ol N(r, \gamma_i;f)+2\sum\limits_{i=m_1+1}^{m_1+m_2} \ol N(r, \gamma_i;f)\\ &\leq& 2\ol N(r,0;f)+ (\Gamma_2-1)T(r,f). \eeas
{\bf Subcase 2.2.} Next suppose one among $\gamma_i's\; (i=m_1+1,m_1+2,\ldots,m_1+m_2)$ is zero. Without loss of generality let us assume that $\gamma_{m_1+1}=0$. Then 
\beas \ol N(r,0; F) &\leq& \ol N(r,0;f)+ \sum\limits_{i=1}^{m_1}\ol N(r, \gamma_i;f)+\sum\limits_{i=m_1+2}^{m_1+m_2}\ol N(r, \gamma_i;f)\\ &\leq&  \ol N(r,0;f)+(\Gamma_1-1) T(r,f);\eeas
\beas N_2(r,0; F)& \leq& 2 \ol N(r,0;f) + \sum\limits_{i=1}^{m_1} \ol N(r, \gamma_i;f)+2\sum\limits_{i=m_1+2}^{m_1+m_2} \ol N(r, \gamma_i;f)\\ &\leq& 2\ol N(r,0;f)+ (\Gamma_2-2)T(r,f). \eeas
Combining all the cases we can write 
\beas \ol N(r,0; F) &\leq& \ol N(r,0;f)+ (\Gamma_1-\chi_0^{r-1})T(r,f);\\
N_2(r,0;F) &\leq& (1+\chi_0^{r-1})\ol N(r,0;f)+ (\Gamma_2-\chi_0^{r-1}-\mu_0^{r-1})T(r,f).\eeas
\end{proof} 
\begin{lem} \label{l8}
Let $F$, $G$ be given by (\ref{e2.1}) and $H\not\equiv 0$. If $F$, $G$ share $(1,2)$ and $f$, $g$ share $(0,0)$, then
\beas &&\frac{r}{2}\{T(r,f)+ T(r,g)\}\\& \leq& 3\ol N(r,0;f)+ \sum\limits_{i=1}^{\Gamma_1-\chi_0^{r-1}}\{\ol N(r,\gamma_i;f)+ \ol N(r, \gamma_i;g)\}\\&&+ 2\{\ol N(r, \infty;f) +  \ol N(r,\infty;g)\}+ S(r,f)+S(r,g).\eeas
\end{lem}
\begin{proof}
Using {\it Lemmas \ref{l1}, \ref{l2} and \ref{l7}} the proof can be carried out in the line of the proof of {\it Lemma 2.5} in \cite[p. 385]{TJM_Ban_10}.
\end{proof}
\begin{lem}	\label{l9} Let $f,\; g \in \mathcal{M}_1(\mathbb{C})$. If $f$, $g$ share $(\{\beta_1, \beta_2\}, 0)$, then $\lambda (f)= \lambda (g)$.	\end{lem}
\begin{proof}
Proof of this lemma can be extracted from the first part of the proof of {\it Theorem 1.3} in \cite{OpenM_Chen_17}.
\end{proof}
\begin{lem} \cite{KAP_YangYi_03} \label{l10}
	Let $a_1$, $a_2$, and $a_3$ be three distinct complex numbers in $\mathbb{C}\cup\{\infty\}$. If two non-constant meromorphic functions $f$ and $g$ share $(a_1,\infty)$, $(a_2,\infty)$, $(a_3,\infty)$ and if the order of $f$ and $g$ is neither an integer nor infinite, then $f \equiv g$.
\end{lem}
\begin{lem} \cite[p. 216]{KAP_YangYi_03} \label{l11}
Let $f$, $g$ be two non-constant meromorphic functions. If $f$ and $g$ share $(0,\infty)$, $(1,\infty)$, $(\infty,\infty)$, $(c,\infty)$, where $c \in \mathbb{C} \setminus \{-1, \frac{1}{2}, 2\}$, then $f \equiv g$.
\end{lem}

\section {Proofs of the theorems}
\begin{proof} [\bf {Proof of Theorem \ref{t1}}]
Since $f$, $g$ share $(S_1,l)$, we see that $F$, $G$ share $(1,l)$.\par
We consider the following cases.\par 
{\bf \ul{Case 1.}} Let $l=2$. \par 
So using {\it Lemmas \ref{l1}, \ref{l3}\; and\; \ref{l7}}, we obtain for $\epsilon>0$
\bea \label{et1.1} rT(r,f) &\leq& N_2(r,0;F)+ N_2(r,0;G) + N_2(r, \infty;F)+ N_2(r,\infty;G)\\&&+ S(r,F)+ S(r,G)\nonumber\\&\leq& (1+\chi_0^{r-1})\{\ol N(r,0;f)+ \ol N(r,0;g)\} +(\Gamma_2-\chi_0^{r-1}-\mu_0^{r-1})\{T(r,f)\nonumber\\ && + T(r,g) \}+ 2 \ol N(r, \infty;f)+ 2\ol N(r, \infty;g)+ S(r,f)+ S(r,g)\nonumber   \\ &\leq& (\Gamma_2-\mu_0^{r-1}+1)\{T(r,f)+ T(r,g) \}+ 2 \ol N(r, \infty;f)+ 2\ol N(r, \infty;g)\nonumber\\ && + S(r,f)+ S(r,g).\nonumber\eea
 Similarly, 
\bea \label{et1.2} rT(r,g) &\leq& (\Gamma_2-\mu_0^{r-1}+1)\{T(r,f)+ T(r,g) \}+ 2 \ol N(r, \infty;f)+ 2\ol N(r, \infty;g)\\ && + S(r,f)+ S(r,g).\nonumber\eea
Combining (\ref{et1.1}) and (\ref{et1.2}) we see that
\beas && r\{T(r,f)+T(r,g)\}\\ &\leq& \bigg\{2\Gamma_2-2\mu_0^{r-1}+6-4\min\{\Theta(\infty,f),\Theta(\infty,g)\}+2\epsilon\bigg\}\{T(r,f)+ T(r,g)\}\\&& +S(r,f)+ S(r,g),\eeas
which contradicts (\ref{t1.1}).\par 
Therefore from {\it Lemma \ref{l3}}, we have either $F \cdot G \equiv 1$ or $F \equiv G$.\par 
Let if possible $F \cdot G \equiv 1$,
i.e., 
\bea \label{et1.3} f\prod_{i=1}^{m_1}(f-\gamma_i)\prod_{i=m_1+1}^{m_1+m_2}(f-\gamma_i)^{n_{i}}.g\prod_{i=1}^{m_1}(g-\gamma_i)\prod_{i=m_1+1}^{m_1+m_2}(g-\gamma_i)^{n_{i}}\equiv 1. \eea
Suppose $z_i$ is a zero of $f-\gamma_i$ for any $i \in \{1, 2, \ldots, m_1\}$ of order $p_i$. From (\ref{et1.3}) we see that $z_i$ is a pole of $g$. Suppose that $z_i$ is a pole of $g$ of order $q_i$. So from (\ref{et1.3}) we obtain $p_i= r q_i$. From this we get $r \leq p_i$. Thus 
\bea \label{et1.4} \ol N(r,\gamma_i;f) \leq \frac{1}{r} N(r, \gamma_i;f) \leq \frac{1}{r}T(r,f)+ S(r,f). \eea 
Since $\gamma_i's$ are distinct for $i=1, 2, \ldots, m_1+m_2$, among $\gamma_1$, $\gamma_2$, \ldots, $\gamma_{m_1}$ or $\gamma_{m_1+1}, \gamma_{m_1+2}, \ldots, \gamma_{m_1+m_2}$ at most one can be zero. Thus we consider the following cases.  \par 
{\bf \ul{Subcase 1.1.}} Suppose none of $\gamma_i's\; (i=1,2,\ldots,m_1+m_2)$ is zero. So if $z_0$ be a zero of $f$ of order $p_0$, then $z_0$ is a pole of $g$ of order $q_0$ such that $p_0= rq_0\geq r$. Therefore 
\beas \ol N(r,0;f) \leq \frac{1}{r}\ol N(r,0;f).\eeas

Now by the Second Fundamental Theorem, using (\ref{et1.4}) and {\it Lemma \ref{l1}}, we get
\beas  \Gamma_1T(r,f)&\leq& \ol N(r,0;f)+ \sum\limits_{i=1}^{\Gamma_1} \ol N(r, \gamma_i;f) + \ol N(r, \infty;f)+ S(r,f),\eeas i.e., 
\beas \Gamma_1T(r,f)&\leq& \bigg(\frac{1+ m_1}{r}+m_2+1\bigg) T(r,f) + S(r,f),\eeas
which is a contradiction for $m_1 > 1$.\par 
{\bf \ul{Subcase 1.2.}} Next let one of $\gamma_i's\; (i=1,2,\ldots,m_1+m_2)$ be zero.\par 
{\bf \ul{Subcase 1.2.1.}} Suppose one among $\gamma_i's\; (i=1,2,\ldots,m_1)$ is zero. Without loss of generality let us assume that $\gamma_1=0$. So if $z_1$ be a zero of $f$ of order $p_0$, then $z_1$ is a pole of $g$ of order $q_0$ such that $2 p_0= rq_0\geq r$. Therefore 
\beas \ol N(r,0;f) \leq \frac{2}{r} \ol N(r,0;f).\eeas
Now by the Second Fundamental Theorem, using (\ref{et1.4}) and {\it Lemma \ref{l1}}, we get
\beas  (\Gamma_1-1)T(r,f)&\leq& \ol N(r,0;f)+ \sum\limits_{i=1}^{\Gamma_1-1} \ol N(r, \gamma_i;f) + \ol N(r, \infty;f)+ S(r,f),\eeas i.e., 
\beas (\Gamma_1-1)T(r,f)&\leq& \bigg(\frac{2+ (m_1-1)}{r}+m_2+1\bigg) T(r,f) + S(r,f),\eeas
which is a contradiction for $m_1 >2$.\par 
{\bf \ul{Subcase 1.2.2.}} Next suppose one among $\gamma_i's\; (i=m_1+1,m_1+2,\ldots,m_1+m_2)$ is zero. Without loss of generality let us assume that $\gamma_{m_1+1}=0$. Then if $z_2$ be a zero of $f$ of order $p_0$, then $z_2$ is a pole of $g$ of order $q_0$ such that $(1+\delta_{m_1+1}) p_0= rq_0\geq r$, where $\delta_{m_1+1}$ is the order of the factor $(f-\gamma_{m_1+1})$. Therefore 
\beas \ol N(r,0;f) \leq \frac{1+\delta_{m_1+1}}{r} \ol N(r,0;f).\eeas
Now by the Second Fundamental Theorem, using (\ref{et1.4}) and {\it Lemma \ref{l1}}, we get
\beas  (\Gamma_1-1)T(r,f)&\leq& \ol N(r,0;f)+ \sum\limits_{i=1}^{\Gamma_1-1} \ol N(r, \gamma_i;f) + \ol N(r, \infty;f)+ S(r,f),\eeas i.e., 
\beas (\Gamma_1-1)T(r,f)&\leq& \bigg(\frac{1+\delta_{m_1+1}+m_1}{r}+(m_2-1)+1\bigg) T(r,f) + S(r,f),\eeas
which is a contradiction for $m_1 > 1$.\par  
Therefore $F \cdot G\not \equiv 1$.\par
Hence, $F\equiv G$, i.e., 
\bea \label{et1.5} \frac{(f-\alpha_1)(f-\alpha_2)\ldots (f-\alpha_r)}{(g-\alpha_1)(g-\alpha_2) \ldots (g-\alpha_r)} \equiv 1.\eea  
As we must have
\beas (\beta_1 - \alpha_1)^2(\beta_1 - \alpha_2)^2 \ldots (\beta_1 - \alpha_r)^2 \not= (\beta_2 - \alpha_1)^2 (\beta_2 - \alpha_2)^2 \ldots (\beta_2 - \alpha_r)^2,\eeas
so $f = \beta_1$ if and only if $g= \beta_1$ and $f = \beta_2$ if and only if $g= \beta_2$ since $f$ and $g$ share $(S_2, 0)$. Again from (\ref{et1.5}) we see that $f$ and $g$ share $(\beta_1, \infty)$, $(\beta_2, \infty)$ and $(\infty, \infty)$ from which in the line of the proof of {\it Lemma \ref{l9}} we can show that the order of $f$ and that of $g$ are equal. Thus the conclusion follows from {\it Lemma \ref{l10}}.\par  
{\bf \ul {Case 2.}} Let $0 \leq l \leq 1$. \par
{\bf \ul{Subcase 2.1.}} Let $H \not \equiv 0$. \par
Let $l=1$. Then using {\it Lemmas \ref{l1}, \ref{l4} and \ref{l7}}, we obtain for $\epsilon>0$ 
\bea \label{et1.6} &&rT(r, f)\\ &\leq& N_2(r,0; F)+N_2(r,\infty; F )+N_2(r,0;G)+N_2(r,\infty;G)\nonumber \\&& +\frac{1}{2}\ol N(r,0; F)+ \frac{1}{2}\ol N(r,\infty; F)+ S(r,F)+ S(r,G)\nonumber\\ &\leq& \bigg(\frac{3}{2}+ \chi_0^{r-1}\bigg)\ol N(r,0;f)+ (1+\chi_0^{r-1}) \ol N(r,0;g) \nonumber\\&&+\bigg(\Gamma_2+\frac{\Gamma_1}{2} -\frac{3 }{2}\chi_0^{r-1}-\mu_0^{r-1}\bigg)T(r,f)+ \bigg(\Gamma_2-\chi_0^{r-1}-\mu_0^{r-1}\bigg)T(r,g)
\nonumber\\&&+ \frac{5}{2}\ol N(r, \infty;f)+ 2\ol N(r, \infty;g)+ S(r,f)+ S(r,g)\nonumber\\ &\leq&\bigg(\Gamma_2+\frac{\Gamma_1}{2}-\frac{1}{2}\chi_0^{r-1}-\mu_0^{r-1}+\frac{3}{2}\bigg)T(r,f)+\bigg(\Gamma_2-\mu_0^{r-1}+1\bigg)T(r,g)\nonumber\\&&+ \frac{5}{2}\ol N(r, \infty;f)+ 2\ol N(r, \infty;g)+ S(r,f)+ S(r,g).\nonumber \eea
Similarly, 
\bea \label{et1.7} rT(r,g) &\leq&\bigg(\Gamma_2+\frac{\Gamma_1}{2}-\frac{1}{2}\chi_0^{r-1}-\mu_0^{r-1}+\frac{3}{2}\bigg)T(r,g)+\bigg(\Gamma_2-\mu_0^{r-1}+1\bigg)T(r,f)\\&&+ \frac{5}{2}\ol N(r, \infty;g)+ 2\ol N(r, \infty;f)+ S(r,f)+ S(r,g).\nonumber \eea
Combining (\ref{et1.6}) and (\ref{et1.7}) we get
\beas &&r\{T(r,f)+ T(r,g)\}\\&\leq& \bigg(2\Gamma_2+\frac{\Gamma_1}{2}-\frac{1}{2}\chi_0^{r-1}-2\mu_0^{r-1}+\frac{5}{2}\bigg)\{T(r,f)+T(r,g)\}+ \frac{9}{2}\{\ol N(r, \infty;f)+ \ol N(r, \infty;g)\}\\&&+S(r,f)+S(r,g)\\&\leq&\bigg\{2\Gamma_2+\frac{\Gamma_1}{2}-\frac{1}{2}\chi_0^{r-1}-2\mu_0^{r-1}+7-\frac{9}{2} \min\{\Theta(\infty,f),\Theta(\infty,g)\}+2\epsilon\bigg\}\{T(r,f)+T(r,g)\}\\&&+S(r,f)+ S(r,g),\eeas
which contradicts (\ref{t1.2}).\par 
Next let $l=0$. Then using {\it Lemmas \ref{l1}, \ref{l5} and \ref{l7}}, we have for $\epsilon>0$ 
\bea \label{et1.8} && rT(r, f)\\ &\leq& N_2(r,0;F)+ N_2(r,0;G)+ 2\ol N(r,0;F)+ \ol N(r,0; G) + N_2(r,\infty;F)\nonumber \\&& +  N_2(r,\infty;G)+ 2 \ol N(r,\infty;F)+ \ol N(r, \infty; G)+ S(r,F) + S(r,G)\nonumber\\
&\leq& (3+\chi_0^{r-1})\ol N(r,0;f)+(2+\chi_0^{r-1})\ol N(r,0;g)+ \bigg(\Gamma_2+2\Gamma_1-3\chi_0^{r-1}-\mu_0^{r-1}\bigg)T(r,f)\nonumber\\&&+\bigg(\Gamma_2+\Gamma_1-2\chi_0^{r-1}-\mu_0^{r-1}\bigg)T(r,g)+ 4 \ol N(r, \infty;f) + 3 \ol N(r,\infty;g)+ S(r,f)+ S(r,g)\nonumber\\
&\leq&\bigg(\Gamma_2+2\Gamma_1-2\chi_0^{r-1}-\mu_0^{r-1}+3\bigg)T(r,f)+\bigg(\Gamma_2+\Gamma_1-\chi_0^{r-1}-\mu_0^{r-1}+2\bigg)T(r,g)\nonumber\\&&+ 4 \ol N(r, \infty;f) + 3 \ol N(r,\infty;g)+ S(r,f)+ S(r,g).\nonumber \eea 
Similarly, 
\bea \label{et1.9}  &&rT(r, g)\\&\leq&\bigg(\Gamma_2+2\Gamma_1-2\chi_0^{r-1}-\mu_0^{r-1}+3\bigg)T(r,g)+\bigg(\Gamma_2+\Gamma_1-\chi_0^{r-1}-\mu_0^{r-1}+2\bigg)T(r,f)\nonumber\\&&+ 4 \ol N(r, \infty;g) + 3 \ol N(r,\infty;f)+ S(r,f)+ S(r,g).\nonumber \eea
Combining (\ref{et1.8}) and (\ref{et1.9}) we get
\beas&&r\{T(r,f)+ T(r,g)\}\\&\leq& \bigg\{2\Gamma_2+3\Gamma_1-3\chi_0^{r-1}-2\mu_0^{r-1}+12-7\min\{\Theta(\infty,f), \Theta(\infty,g)\}+ 2\epsilon\bigg\}\{T(r,f)+T(r,g)\}\nonumber\\&&+ S(r,f)+ S(r,g), \eeas
which contradicts (\ref{t1.3}).\par 
{\bf \ul {Subcase 2.2}} Let $H \equiv 0$. Then $F$, $G$ share $(1,\infty)$.\par
So by {\it Lemma \ref{l6}}, we get $T(r,G)= T(r,F)+ O(1)$.\par 
Now in view of {\it Lemmas \ref{l1} and \ref{l7}}, we see that for $\epsilon>0$\beas &&\ol N(r,0;F)+ \ol N(r,0;G)+ \ol N(r,\infty;F) + \ol N(r,\infty;G)\\&\leq& \ol N(r,0;f)+ \ol N(r,0;g)+(\Gamma_1-\chi_0^{r-1})\{T(r,f)+T(r,g)\}\\&&+ \ol N(r,\infty;f)+\ol N(r,\infty;g)+ S(r,f)+S(r,g)\\ &\leq& \{2\Gamma_1-2\chi_0^{r-1}+ 4-\Theta(\infty,f)- \Theta(\infty,g)+ \epsilon\}T(r)+S(r)\\ &\leq& \frac{2\Gamma_1-2\chi_0^{r-1}+ 4-\Theta(\infty,f)- \Theta(\infty,g)+ \epsilon}{r}T_1(r)+S(r), \eeas where $T_1(r) = \max\{T(r,F), T(r,G)\}$.
As $r>2\Gamma_1-2\chi_0^{r-1}+4-\Theta(\infty,f)- \Theta(\infty,g)$, using {\it Lemma \ref{l6}} we obtain either $F\equiv G$ or $F \cdot G \equiv 1$. Therefore by the similar argument as in {\it Case 1}, we get $f \equiv g$. \par  
\end{proof}
\begin{proof} [\bf {Proof of Corollary \ref{c1}}]
Since $f$, $g$ share $(S_1,l)$, we see that $F$, $G$ share $(1,l)$.\par
We consider the following cases.\par 
{\bf \ul{Case 1.}} Let $l=2$. \par
Putting $\Theta(\infty,f)= \Theta(\infty,g)=1$ and proceeding in the line of the proof of {\it Case 1} in {\it Theorem \ref{t1}}, we get $r\leq 2\Gamma_2-2\mu_0^{r-1}+2$, which contradicts (\ref{c1.1}). Hence either $F\equiv G$ or $F \cdot G \equiv 1$.\par 
We see that as $f$, $g$ $\in \mathcal{M}_1(\mathbb{C})$, hence $F$, $G$ $\in \mathcal{M}_1(\mathbb{C})$. \par 
If possible let $F \cdot G \equiv 1$, i.e., $U(f) \cdot U(g)\equiv 1$. Now as done in \cite[p. 337]{AUSM_SahKar_18} we arrive at a contradiction and hence the case $F \cdot G \equiv 1$ cannot occur. 	
So $F \equiv G$. Therefore by the similar argument as in {\it Case 1} of the proof of {\it Theorem \ref{t1}}, we get $f \equiv g$. \par
{\bf \ul{Case 2.}} Let $0\leq l \leq 1$.\par 
 Putting $\Theta(\infty,f)= \Theta(\infty,g)=1$ and proceeding in the line of the proof of {\it Case 2} of {\it Theorem \ref{t1}}, we get $r \leq  2 \Gamma_2 +\frac{\Gamma_1}{2}-\frac{1}{2}\chi_0^{r-1}-2 \mu_0^{r-1}+\frac{5}{2}$ and $r \leq 2\Gamma_2+ 3\Gamma_1-3\chi_0^{r-1}-2\mu_0^{r-1}+5$ for $l=1$ and $l=0$ respectively which contradicts (\ref{c1.2}) and (\ref{c1.3}). Hence for both $l=1$ and $l=0$ we have either $F\equiv G$ or $F \cdot G\equiv 1$. Next following the same way as done in {\it Case 1} of this corollary, the conclusion can be obtained.
\end{proof}
\begin{proof} [\bf {Proof of Theorem \ref{t2}}]
{\bf \ul{Case 1.}} Let $H\not\equiv 0$.
Then using {\it Lemma \ref{l8}}, we see that for $\epsilon>0$ \beas \bigg\{r-2\Gamma_1+2\chi_0^{r-1}-7+4\min\{\Theta(\infty,f),\Theta(\infty,g)\}-2\epsilon\bigg\}\{T(r,f)+T(r,g)\}\leq S(r,f)+S(r,g),\eeas
which contradicts the assumption (\ref{t2.1}).\par 
{\bf \ul{Case 2.}} Let $H\equiv 0.$
Then as done in {\it Subcase 2.2} of {\it Theorem \ref{t1}}, 
we see that 
\beas \frac{\ol N(r,0;F)+ \ol N(r,0;G)+ \ol N(r,\infty;F) + \ol N(r,\infty;G)}{T_1(r)}<1 \eeas for $r>2\Gamma_1-2\chi_0^{r-1}+2$, which clearly holds from the condition (\ref{t2.1}). So using {\it Lemma \ref{l6}}, we get either $F\equiv G$ or $F \cdot G\equiv 1$.\par 
If possible let $F \cdot G \equiv 1$.\par
Since $f$, $g$ share $(0,0)$, it follows that $0$ is the Picard exceptional value of $f$ and $g$. 
Then in view of (\ref{t2.2}), and using the similar argument as in {\it Case 1} of the proof of {\it Theorem \ref{t1}}, we get a contradiction. Hence $F \cdot G \not \equiv 1.$ \par 
So, $F \equiv G$. Now as done in {\it Case 1} of {\it Theorem \ref{t1}} under the assumption of the condition (\ref{t2.3}) we obtain that $f$, $g$ share $(1,  \infty)$, $(c, \infty)$ and $(\infty, \infty)$. Again $F \equiv G$ implies $fQ(f)\equiv g Q(g)$. As $f$ and $g$ share $(0,0)$ we note that they share $(0, \infty)$. Hence by {\it Lemma \ref{l11}}, we get $f\equiv g$. 
\end{proof}
\begin{proof}[\bf {Proof of Corollary \ref{c2}}]
As $f$ and $g$ are transcendental meromorphic functions, so $\ol N(r,\infty;f)$ $= S(r,f)$ and $\ol N(r, \infty; g) = S(r,g)$,
which implies that $F \cdot G  \equiv 1$ is not possible under the condition (\ref{c2.2}). \par
The rest of the proof can be carried out in the line of proof of {\it Theorem \ref{t2}}.
\end{proof}
\section{Compliance with Ethical Standards}
{\bf Conflict of interest} The authors declare that there are no conflicts of interest regarding the publication of this paper.\par
{\bf Human/animals participants} The authors declare that there is no research involving human participants and/or animals in the contained of this paper. 

\end{document}